\documentclass{proc-l-hijacked}
\usepackage{amssymb, amsmath, amsthm, hyperref,}

\newtheorem{theorem}{Theorem}[section]
\newtheorem{lemma}[theorem]{Lemma}

\newtheorem{proposition}[theorem]{Proposition}
\newtheorem{corollary}[theorem]{Corollary}

\theoremstyle{definition}

\DeclareMathOperator{\Soc}{Soc}
\DeclareMathOperator{\rank}{rank}
\DeclareMathOperator{\comm}{comm}
\DeclareMathOperator{\tr}{tr}
\DeclareMathOperator{\spann}{span}
\DeclareMathOperator{\Rad}{Rad}
\DeclareMathOperator{\QN}{QN}

\def\qne{quasinilpotent equivalent}
\def\ro{{\mathcal F}_1\,}
\def\fr{\mathcal F\,}

\numberwithin{equation}{section}

\begin{document}

	\title[]{Finite spectra and quasinilpotent equivalence in Banach algebras}
	\author{R. Brits and H. Raubenheimer}
	\address{Department of Mathematics, University of Johannesburg, South Africa}
	\email{rbrits@uj.ac.za, heinrichr@uj.ac.za }

\begin{abstract}
This paper further investigates the implications of
quasinilpotent equivalence for (pairs of) elements belonging to the socle of a semisimple Banach algebra. Specifically, not only does quasinilpotent equivalence of two socle elements imply spectral equality, but also the trace, determinant and spectral multiplicities of the elements must agree. It is hence shown that quasinilpotent equivalence is established by a weaker formula (than that of the spectral semidistance). More generally, in the second part, we show that two elements possessing finite spectra are quasinilpotent equivalent if and only if they share the same set of Riesz projections. This is then used to obtain further characterizations in a number of general, as well as more specific situations. Thirdly, we show that the ideas in the preceding sections turn out to be useful in the case of $C^*$-algebras, but now for elements with infinite spectra; we give two results which may indicate a direction for further research.
\end{abstract}
	\parindent 0mm
	
	\maketitle

\section{Introduction}\label{sec1}

This paper is a continuation of work done in \cite{raub} and \cite{razpet}. The notion of quasinilpotent equivalence for linear operators is due to Colojoar\v a and  Foia\c s ~\cite{c+f, c+f2}. Their ideas have been extended to general Banach algebras by Razpet in ~\cite{razpet}.

Throughout this paper $A$ is a Banach algebra with unit $1$ over the field $\mathbb C$ of complex numbers. The spectrum of $a\in A$ will be denoted by $\sigma (a,A)$, the ``nonzero" spectrum, $\sigma (a,A)\backslash\{0\}$, by $\sigma^\prime(a,A)$, and the spectral radius of  $a\in A$ by $r(a,A)$. Whenever there is no ambiguity we shall drop the $A$ in $\sigma$ and $r$. An element $a\in A$ is said to be \emph{quasinilpotent} if $\sigma (a)=\{0\}$, equivalently $\lim_n\| a^n\|^{1/n}=0$. The set of quasinilpotent elements is denoted by $\QN(A)$, and the Jacobson radical, a subset of $\QN(A)$, by $\Rad(A)$. If $x\in A$, then $\comm(x)$ is the commutant of $x$. For $n\in\mathbb N$ we denote the algebra consisting of all $n\times n$ complex matrices by $M_n(\mathbb C)$. Finally, if $x\in A$ then by convention $x^0=1$.
\par
For each $a,b\in A$ associate operators $L_a,R_b$ and $C_{a,b}$ acting on $A$ by
the relations
\[L_ax=ax, \quad R_bx=xb\quad \hbox{and}\quad C_{a,b}x=(L_a-R_b)x\]
for all $x\in A$. It is easy to see that $L_a,R_b$ and $C_{a,b}$  are bounded linear operators on $A$, i.e., $L_a,R_b,C_{a,b}\in{\mathcal L} (A)$.

\section{Quasinilpotent equivalence}\label{sec2}

Let $a,b\in A$. Since the operators $L_a$ and $R_b$ commute
\begin{eqnarray}
C_{a,b}^nx=\sum_{k=0}^n(-1)^k{n \choose k}a^{n-k}xb^k
\end{eqnarray}
for all $x\in A$. We have
\begin{eqnarray}\label{comb}
C_{a,b}^{n+1}x=a(C_{a,b}^nx)-(C_{a,b}^nx)b,
\end{eqnarray}
and also, if $c\in A$, one can prove ~(\cite{f+v})
\begin{eqnarray}\label{xy}
C_{a,b}^{n}(xy)=\sum_{k=0}^n{n \choose k}(C_{a,c}^{n-k}x)(C_{c,b}^ky)
\end{eqnarray}
for all $x,y\in A$. A straightforward computation shows that for each $\lambda\in\mathbb C$
\begin{eqnarray}\label{exp}
e^{{\lambda}a}e^{-{\lambda}b}=\sum_{j=0}^\infty\frac{{\lambda}^jC_{a,b}^j1}{j!}.
\end{eqnarray}
Define $\rho:A\times A\rightarrow \mathbb R$ by
\begin{eqnarray}\label{rhoab}
\rho (a,b)=\limsup_n\| C_{a,b}^n1\| ^{1/n}.
\end{eqnarray}
We note  that the precise relationship between $\rho (a,b)$ and $\rho (b,a)$ is apparently unknown. If, however, $a$ and $b$ commute then by \eqref{rhoab} $\rho (a,b)=\rho (b,a)=r(a-b)$.
Now define
\begin{eqnarray}
d(a,b)=\max \{\rho (a,b),\rho (b,a)\}.
\end{eqnarray}
The identity \eqref{xy} can be used to prove that the function $d$ is a semimetric on $A$, called the \emph{spectral semidistance }from $a$ to $b$. In general $d$ is not a metric on $A$, with pathologies already evident on pairs of elements belonging to $\QN(A)$ (see the remarks preceding Proposition 2.2 in ~\cite{razpet} or Proposition~\ref{quasi} of the present paper). Following ~\cite{razpet}, elements $a,b\in A$ are called \emph{\qne} if $d(a,b)=0$. If $a$ and $b$ are quasinilpotent equivalent then $\sigma(a)=\sigma(b)$  (see \cite[Theorem 2.1]{raub}), a property which will be used throughout this paper without further reference. The semimetric $d$ seems worthwhile studying because it is really an extension of the spectral radius: For each $a\in A$ we have $d(a,0)=r(a)$.

Proposition~\ref{quasi} is a generalization of the fact that any two elements belonging to $\QN(A)$ are quasinilpotent equivalent \cite{raub}. Lemmas ~\ref{qc} and ~\ref{powers} will be used in forthcoming results.
\begin{proposition}\label{quasi} Let $A$ be a Banach algebra, $a\in A$, and suppose $q\in A$ is quasinilpotent. Then
	$$\rho(a,q)=\rho(q,a)=r(a).$$
\end{proposition}
\begin{proof} It follows from ~\cite[p. 380 (6)]{razpet}  that
	$$\rho(a,q)\leq\rho(a,0)+\rho(0,q)=\rho(a,0)\leq\rho(a,q)+\rho(q,0)=\rho(a,q).$$
	In a similar way $\rho(q,a)=r(a)$.
\end{proof}

\begin{lemma}\label{qc} Let $A$ be a Banach algebra, $a,b\in A$, and suppose $q,r\in\QN(A)$ commute with $a$ and $b$ respectively. Then
	$$\rho(a+q,b+r)=\rho(a,b)$$
\end{lemma}
\begin{proof} Again from ~\cite[p. 380 (6)]{razpet}  we have that
	\begin{align*}
	\rho(a+q,b+r)&\leq\rho(a+q,a)+\rho(a,b+r)\\&=\rho(a,b+r)\\&\leq
	\rho(a,b)+\rho(b,b+r)\\&=\rho(a,b).
	\end{align*}
	It therefore also follows that
	$$\rho(a,b)=\rho((a+q)-q,(b+r)-r)\leq\rho(a+q,b+r).$$
\end{proof}

\begin{lemma}\label{powers} If $\rho(a,b)=0$ then $\rho(a^k,b^k)=0$ for each $k\in\mathbb N$.
\end{lemma}
\begin{proof} Fix $k\in\mathbb N$.
	Since $L_a$ and $R_b$ commute factorization gives $$C_{a^k,b^k}=L_{a^k}-R_{b^k}=L^k_a-R^k_b=P(L_a,R_b)\Bigl(L_a-R_b\Bigr)$$
	where $P:=P(L_a,R_b)$ is an operator commuting with $L_a$ and $R_b$. It thus follows that
	$\rho(a^k,b^k)\leq r(P)\rho(a,b)$.
\end{proof}

\section{Finite rank elements and quasinilpotent equivalence}\label{sec3}

In this section we shall require that $A$ be a \emph{semiprime} Banach algebra, i.e., $xAx=\{0\}$ implies $x=0$ holds for all $x\in A$. It can be shown that all semisimple Banach algebras are semiprime. Following Puhl ~\cite{puhl} we call an element $0\not=a\in A$ \emph{rank one} if $aAa\subseteq {\mathbb C}a$. Denote the set of these elements by $\ro$. By ~\cite[Lemma 2.7]{puhl} $ \ro A,\, A\ro=\,\ro.$ A projection (idempotent) belonging to $\ro$ is called a \emph{minimal projection}. Let $\fr$ denote the set of all $u\in A$ of the form $u=\sum_{i=1}^nu_i$ with $u_i\in \ro$. We call $\fr $the set of \emph{finite rank} elements of $A$. $\fr$ is a two sided ideal in $A$ and it coincides with the socle of $A$, i.e., $ \Soc(A)\,=\,\fr$. If $a\in\Soc(A)$, then $\sigma(a)$ is a finite set and hence, corresponding to $\alpha\in\sigma(a)$, one can find a small circle $\Gamma_{\alpha}$ isolating $\alpha$ from the remaining spectrum of $a$. We denote by
\[p(\alpha,a)=\frac{1}{2{\pi}i}\int_{\Gamma_{\alpha}}(\lambda -a)^{-1}\,d\lambda\] the \emph{Riesz projection} associated with $a$ and $\alpha$. If $\alpha\notin\sigma(a)$, then, by Cauchy's Theorem, $p(\alpha,a)=0$.  Recall that $p(\alpha,a)$ belongs to the bicommutant of $a$.
\par
For another approach to rank one and finite rank elements see ~\cite{a+m,harte}; if $A$ is a semisimple Banach algebra then the notion of rank one and finite rank elements in the sense of Puhl ~\cite{puhl} coincides with the notion of rank one and finite rank elements in the sense of Aupetit and Mouton (~\cite[Theorem 4]{harte} and ~\cite[Theorem 2.12]{a+m}).
\par
Let $A$ be a semiprime Banach algebra and $a,b\in A$. Suppose $a,b\in\ro$ and $d(a,b)=0$. If $a\in\,$QN($A$) then $\sigma (a)=\sigma (b)=\{0\}$. In view of ~\cite[Section 2 and Lemma 2.8]{puhl} $a^2=b^2=0$. If we suppose $a,b\in\fr$, $d(a,b)=0$ and $a\in\,$QN($A$), then again $\sigma (a)=\sigma (b)=\{0\}$. Now, in view of ~\cite[Lemma 3.10]{m+r}, there is a natural number $m$ such that $a^m=b^m=0$.\par
Let $A$ be a semisimple Banach algebra and $a\in A$. Following Aupetit and Mouton  we define the \emph{rank} of $a$ by

\begin{equation}\label{rank}
\rank(a)=\sup_{x\in A}\# \sigma^\prime(xa)\leq\infty.
\end{equation}
where the symbol $\#K$ denotes the number of distinct elements in a set $K\subseteq \mathbb C$. With respect to ~\eqref{rank} recall that Jacobson's Lemma says $\sigma^\prime(xa)=\sigma^\prime(ax)$.
If $x\in A$ is such that $\#\sigma^\prime(xa)=\rank(a)$, then we say \emph{$a$ assumes its rank at $x$}. Useful
in this regard is the fact that, for each $a\in\Soc(A)$, the set
\begin{equation}\label{assume}
E(a)=\{x\in A:\#\sigma^\prime(xa)=\rank(a)\}
\end{equation}
is dense and open in $A$.
For $a\in\Soc(A)$, Aupetit and Mouton define the \emph{trace} and \emph{determinant} as:
\begin{equation}\label{trace}
\tr(a)=\sum\limits_{\lambda\in\sigma(a)}\lambda\, m(\lambda,a)
\end{equation}
\begin{equation}\label{determinant}
\det(1+a)=\prod\limits_{\lambda\in\sigma(a)}(1+\lambda)^{m(\lambda,a)}
\end{equation}
where $m(\lambda,a)$ is the \emph{multiplicity of $a$ at $\lambda$}. A brief description of the notion of multiplicity in the abstract case goes as follows (for particular details one should consult \cite{a+m}): Let $a\in\Soc(A)$, $\lambda\in\sigma(a)$ and let $V_\lambda$ be an open disk centered at $\lambda$ such that $V_\lambda$ contains no other points of $\sigma(a)$. It can be shown \cite[p.119--120]{a+m} that there exists an open ball, say $U\subset A$, centered at $1$ such that $\#\left[\sigma(xa)\cap V_\lambda\right]$ is constant as $x$ runs through $E(a)\cap U$. This constant integer is the multiplicity of $a$ at $\lambda$.

In the operator case, $A=\mathcal L(X)$, where $X$ is a Banach space, the ``spectral" rank, trace and determinant all coincide with the respective classical operator definitions.

To develop the results in this section we need the following basic:
\begin{lemma}\label{lemma}
	Let $X$ be a finite dimensional normed space with basis $\{e_1,\ldots ,e_k\}$. If $T:X\to X$ is a linear operator such that for each $j$
	\[\limsup_n\| T^n e_j\|^{1/n}=0,\]
	then $T$ is nilpotent.
\end{lemma}
\begin{proof} Since dim$X\,<\infty$ there is a constant $c>0$ such that for each $x=\alpha_1e_1+\ldots +\alpha_ke_k$ with $\|x\|\leq 1$ one has that $|\alpha_1|+\ldots +|\alpha_k|\leq c$. Thus
\begin{eqnarray*}
	\limsup_n[\sup_{\|x\|\leq 1}\| T^nx\|^{1/n}]&\leq&\limsup_n[c^{1/n}\max_j\| T^ne_j\| ^{1/n}]\\
	&\leq&\sum_{j=0}^k\limsup_n[c^{1/n}\| T^ne_j\| ^{1/n}]\\                                                                                                                                 &=&0.
\end{eqnarray*}
So the hypothesis implies that
\begin{eqnarray*}
	\lim_{n\to\infty}\| T^n\|^{1/n}=0,
\end{eqnarray*}
which means that $T$ is quasinilpotent. But, since dim$\,X<\infty$, $T$ is in fact nilpotent.
\end{proof}
As one would expect, Lemma~\ref{lemma} can also be proved via spectral arguments, avoiding the norm altogether. Regarding Lemma \ref{lemma}, recall that it follows from local spectral theory, that if for arbitrary $x$ in a Banach space $X$, $\|T^nx\|^{1/n}\to 0$, then $T\in {\mathcal L}(X)$ is quasinilpotent \cite[Corollary 34.5]{muller}. We proceed to show that, for elements belonging to $\Soc(A)$, quasinilpotent equivalence is implied by the formally weaker requirement $\rho(a,b)=0$. Moreover, if this is the case, then one does not merely have $\sigma(a)=\sigma(b)$, but also that the multiplicities, $m(\lambda,a)$ and $m(\lambda,b)$, corresponding to nonzero spectral points $\lambda$ coincide. We first need to establish

\begin{theorem}\label{tracethm} Let $A$ be a semisimple Banach algebra and $a,b\in  \Soc(A)$. Then \[\rho(a,b)=0\Rightarrow \tr(a)=\tr(b).\]
\end{theorem}
\begin{proof} Since each element of $ \Soc(A)$ is algebraic it follows that the set
	$L=\{a^mb^n:m,n\in\mathbb Z^+\}$ spans a finite dimensional vector space containing $1$, $a$ and $b$. Denote $X=\spann L$. It is clear that
	the linear operator $C_{a,b}$ maps $X$ into $X$. Let $Y$ be the subspace of $X$ spanned by the orbit $\{C^n_{a,b}1:n\in\mathbb Z^+\}$.   Then, using Lemma~\ref{lemma}, the hypothesis $\rho(a,b)=0$ implies that $C_{a,b}$ is a nilpotent operator on $Y$. It follows that there is $N\in\mathbb N$ such that
	\begin{equation}\label{ef}
	e^{\lambda a}e^{-\lambda b}=\sum_{j=0}^N \frac{\lambda^j C_{a,b}^j1}{j\,!}.
	\end{equation}
	
	Notice that
	\begin{align*}
	\det(e^{\lambda a}e^{-\lambda b})&=\det(e^{\lambda a})\det(e^{-\lambda b})\\&
	=e^{\tr(\lambda a)}e^{\tr(-\lambda b)}\\&
	=e^{\lambda\tr(a-b)}.
	\end{align*}
	Write $p(\lambda)=\sum_{j=1}^N \frac{\lambda^j C_{a,b}^j1}{j\,!}$ and observe that, for each $\lambda\in\mathbb C$,
	\[\rank(p(\lambda))\leq\sum_{j=1}^N\rank(C_{a,b}^j1):=M<\infty.\]
	Combining this with \eqref{ef} we have
	\begin{align*}|e^{\lambda\tr(a-b)}|&=|\det(1+p(\lambda))|\\&
	\leq[1+r(p(\lambda)]^M\\&
	\leq[1+\|(p(\lambda)\|]^M\\&
	\leq[1+\sum\limits_{j=1}^N\frac{\|C_{a,b}^j1\|}{j\,!}|\lambda|^j]^M,
	\end{align*}
	which implies that the entire function $\lambda\mapsto e^{\lambda\tr(a-b)}$ has polynomial growth, and must therefore be a polynomial. It thus follows that $\lambda\mapsto e^{\lambda\tr(a-b)}$ is in fact constantly $1$ and hence $\tr(a-b)=0$ which completes the proof.
\end{proof}

\begin{corollary}\label{specreq}
	Let $A$ be a semisimple Banach algebra and let $a,b\in  \Soc(A)$. If $\rho(a,b)=0$ then $r(a)=r(b)$.
\end{corollary}

\begin{proof}
	If $\rho (a,b)=0$, then, by Lemma~\ref{powers}, $\rho (a^k,b^k)=0$ for each $k\in\mathbb N$. Theorem~\ref{tracethm} and \cite[Theorem 3.5]{a+m} imply
	\[r(a)=\limsup_k|\tr (a^k)|^{\frac{1}{k}}=\limsup_k|\tr (b^k)|^{\frac{1}{k}}=r(b).\]
\end{proof}

In view of the next result Corollary~\ref{specreq} is short lived.

\begin{corollary}\label{equal} Let $A$ be a semisimple Banach algebra and let $a,b\in  \Soc(A)$. Then \[\rho(a,b)=0\Leftrightarrow\rho(b,a)=0.\]
\end{corollary}

\begin{proof} Suppose $\rho(a,b)=0$. The first part of the proof of Theorem~\ref{tracethm} shows that $q(\lambda)=e^{\lambda a}e^{-\lambda b}$ is a polynomial in $\lambda$ with coefficients belonging to $A$. Thus, to establish the result, it suffices to show that $q^{-1}(\lambda)$ is also a polynomial. Let $B$ be the Banach algebra generated by $\{1,a,b\}$. Then $B$ is finite dimensional, but not necessarily semisimple. Denote $\tilde B=B/\Rad(B)$ and for each $x\in B$, by $\tilde x$ the image of $x$ under the canonical homomorphism $C:B\rightarrow \tilde B$. Since $\tilde B$ is now semisimple it follows that there is a least integer $N_0$ such that $\tilde B$ is a (generally non-surjective) algebra embedding into $M_{N_0}(\mathbb C)$. In this way we may view $\tilde B$ as a closed unital subalgebra of $M_{N_0}(\mathbb C)$. The polynomial $\tilde q(\lambda)$ is therefore a $N_0\times N_0$  matrix whose entries, say $\tilde q_{i,j}(\lambda)$, are polynomials in $\lambda$ with coefficients belonging to $\mathbb C$. Moreover, for each $\lambda\in\mathbb C$, $\tilde q(\lambda)$ is invertible in $M_{N_0}(\mathbb C)$. We now calculate $$\tilde q^{-1}(\lambda)=\frac{1}{\det(\tilde q(\lambda))}b(\lambda)$$ where $b(\lambda)$ is a $N_0\times N_0$ matrix depending analytically on $\lambda$. Since its $(i,j)$ entry is the $(j,i)$ cofactor of $\tilde q(\lambda)$, and $\tilde q(\lambda)$ is a polynomial, it follows that $b(\lambda)$ is a polynomial. But from Theorem~\ref{tracethm} we get that
	$\det(\tilde q(\lambda))=\det(q(\lambda))=1$ for each $\lambda\in\mathbb C$, whence it follows that $\tilde q^{-1}(\lambda)=b(\lambda)$ is a polynomial. Returning to the algebra $B$, we now have the following: There exists a polynomial $p(\lambda)\in B$ such that
	$q^{-1}(\lambda)=p(\lambda)+r(\lambda)$ where, for each $\lambda$, $r(\lambda)\in\Rad(B)$. From this one obtains $p(\lambda)q(\lambda)+r(\lambda)q(\lambda)=1$ where $r(\lambda)q(\lambda)\in\Rad(B)$. Since $\dim B<\infty$ there is $M\in\mathbb N$ (independent of $\lambda$) such that
	$[1-p(\lambda)q(\lambda)]^M=0$ for each $\lambda\in\mathbb C$. The binomial expansion on the left then yields
	\begin{equation}\label{be}
	1=\sum\limits_{j=1}^M(-1)^{j+1}{M \choose j}[p(\lambda)q(\lambda)]^j.
	\end{equation}
	Finally, multiplication of \eqref{be} throughout by $q^{-1}(\lambda)$ on the right shows that $q^{-1}(\lambda)$ is a polynomial.
\end{proof}

\begin{corollary} Let $A$ be a semisimple Banach algebra and $a,b\in  \Soc(A)$.
	Then \[\rho(a,b)=0\Rightarrow \det(1+\lambda a)=\det(1+\lambda b),\ \ \lambda\in\mathbb C.\]
\end{corollary}
\begin{proof} If $1\in\sigma(-\lambda a)=\sigma(-\lambda b)$ the result follows trivially. So assume the contrary. Since elements belonging to the socle have discrete spectrum $\log(1+\lambda a)$ and $\log(1+\lambda b)$ exist in $A$. Furthermore, if $\lambda\in U:=\{\lambda: \|\lambda a\|<1 \mbox{ and }\|\lambda b\|<1\}$ then \[\log(1+\lambda a)=\sum\limits_{j=1}^\infty\frac{1}{j}(-\lambda a)^{j+1}
	\ \mbox{ and }\ \log(1+\lambda b)=\sum\limits_{j=1}^\infty\frac{1}{j}(-\lambda b)^{j+1}.\] Combining this with Lemma~\ref{powers} and Theorem~\ref{tracethm}, the linearity of the trace implies
	that $\tr(\log(1+\lambda a))=\tr(\log(1+\lambda b))$ and hence
	$$\det(1+\lambda a)=e^{\tr(\log(1+\lambda a))}=e^{\tr(\log(1+\lambda b))}=\det(1+\lambda b)$$ for all $\lambda\in U$. Since the entire functions $\lambda\mapsto\det(1+\lambda a)$ and $\lambda\mapsto\det(1+\lambda b)$ agree on the open set $U$ they also agree on $\mathbb C$.
\end{proof}

\begin{corollary}\label{multi} Let $A$ be a semisimple Banach algebra and let $a,b\in  \Soc(A)$. If $\rho(a,b)=0$ then $m(\lambda,a)=m(\lambda,b)$ for each $\lambda\in\sigma^\prime(a)=\sigma^\prime(b)$.
\end{corollary}
\begin{proof} With the hypothesis we can write $\sigma^\prime(a)=\sigma^\prime(b)=\{\lambda_1,\dots,\lambda_k\}$. Thus
	\[\det(1+\lambda a)=\prod\limits_{i=1}^k(1+\lambda\lambda_i)^{m(\lambda_i,a)}=
	\prod\limits_{i=1}^k(1+\lambda\lambda_i)^{m(\lambda_i,b)}=\det(1+\lambda b)\] holds for all $\lambda\in\mathbb C$. But if ${m(\lambda_j,a)}\not={m(\lambda_j,b)}$ for some $j$, then the order of the root $-\lambda_j^{-1}$ of $\det(1+\lambda a)$ is not uniquely determined which is absurd.
\end{proof}

\begin{corollary}\label{traceappl} Let $A$ be a semisimple Banach algebra. If $a,b\in  \Soc(A)$ then $a=b$ if and only $\rho(xa,xb)=0$ for all $x$ in an arbitrary small open subset $N$ of $ \Soc(A)$.
\end{corollary}
\begin{proof} For the reverse implication: Fix $y\in N$ and let $x\in  \Soc(A)$ be arbitrary. There exists $\epsilon>0$ such that $\rho((y+\lambda x)a,(y+\lambda x)b)=0$ for $|\lambda|<\epsilon$.  The linearity of the trace, together with Theorem~\ref{tracethm}, gives $\tr(xa)=\tr(xb)$ and the result is immediate from \cite[Corollary 3.6]{a+m}.
\end{proof}

Since the definition of multiplicity, $m(\lambda,a)$,  involves products of $a$ with other elements $x\in A$ it is interesting that $\rho(a,b)=0$ can establish the conclusion in Corollary~\ref{multi}; after all, the expression $\rho(a,b)$ concerns only the elements $a$ and $b$. On the other hand, quasinilpotent equivalence of $a$ and $b$ cannot, in the general sense, establish any connection between the respective ranks of $a$ and $b$; if $a,b\in\QN(A)$ then irrespective of rank we have $\rho(a,b)=0$. Theorem~\ref{char1} in the following section clarifies these issues.

\section{Riesz projections and quasinilpotent equivalence}\label{sec4}
The main result of \cite{razpet} says that any two quasinilpotent equivalent elements, which are simultaneously roots of an entire function, $f$, possessing only simple zeros, must necessarily be equal. Inspection of the proof reveals that the (rather strong) assumptions concerning the function $f$, somewhat obscures a useful consequence of quasinilpotent equivalence which involves the Riesz projections associated with isolated spectral values. We show here, for the case of elements with finite spectra, how this can be used to derive spectral-algebraic characterizations of quasinilpotent equivalence.

A modification of the proof of \cite[Theorem 3.1]{razpet} yields Lemma~\ref{Riesz}
which is also used later to obtain the main result in Section~\ref{sec5}.

\begin{lemma}\label{Riesz} Let $A$ be a Banach algebra. Suppose $d(a,b)=0$ and that $\lambda_1\not=0$ is an isolated point of $\sigma(a)=\sigma(b)$. If the Riesz projections $p_1=p(\lambda_1,a)$ and $q_1=p(\lambda_1,b)$ satisfy $ap_1=\lambda_1p_1$ and $bq_1=\lambda_1q_1$, then $p_1=q_1$.
\end{lemma}
\begin{proof} Define \[F(\lambda)=\sum_{r=0}^\infty\frac{C_{b,a}^r1}{(\lambda-\lambda_1)^{r+1}}\] which converges for $\lambda\not=\lambda_1$ since $d(a,b)=0$. Using the identity
	\begin{equation}\label{iden}
	(\lambda -b)C_{b,a}^r1=(C_{b,a}^r1)(\lambda -a)-C_{b,a}^{r+1}1,
	\end{equation}
	together with $ap_1=\lambda_1p_1$,  observe that the series given by
	\[(\lambda -b)F(\lambda)p_1=\sum_{r=0}^\infty\frac{C_{b,a}^r1(\lambda -a)-C_{b,a}^{r+1}1}{{(\lambda -\lambda_1)^{r+1}}}{p_1}\]
	is telescopic whence we obtain
	\[(\lambda -b)F(\lambda)p_1=p_1.\]
	If $\Gamma_1$ is a small circle disjoint from $\sigma(a)$, and surrounding only $\lambda_1\in\sigma(a)$ then
	\begin{equation}\label{int}
	p_1=\frac{1}{2{\pi}i}\int_{\Gamma_1}F(\lambda )p_1\,d\lambda =\frac{1}{2{\pi}i}\int_{\Gamma_1}(\lambda -b)^{-1}p_1\,d\lambda =q_1p_1.
	\end{equation}
	Now define  \[G(\lambda)=\sum_{r=0}^\infty\frac{C_{b,a}^r1}{(\lambda_1-\lambda)^{r+1}}\] which also converges for $\lambda\not=\lambda_1$. A rearrangement of the identity \eqref{iden} then gives
	\[q_1G(\lambda)(a-\lambda)=\sum_{r=0}^\infty q_1\frac{(b-\lambda)C_{b,a}^r1-C_{b,a}^{r+1}1}{{(\lambda_1 -\lambda)^{r+1}}}\]
	so that, similar to the case for $F$,
	\[q_1G(\lambda)(a-\lambda)=q_1.\]
	Multiplication by $(a-\lambda)^{-1}$ followed by integration along $\Gamma_1,$ as in \eqref{int}, yields
	\[-q_1=-q_1p_1\] and the conclusion follows.
\end{proof}

In the case of finite spectra, quasinilpotent equivalence implies equality of the Riesz projections $p(\lambda_j,a)$ and $p(\lambda_j,b)$ even without the additional conditions required in Lemma~\ref{Riesz}; the simple idea here is to show that quasinilpotent equivalence of $a$ and $b$ implies quasinilpotent equivalence of two related but possibly different elements, say $\tilde a$ and $\tilde b$, which have the same spectra, as well as the same sets of  Riesz projections corresponding to $a$ and $b$ respectively, \emph{but}, for which the requirements in Lemma~\ref{Riesz} do hold for all the Riesz projections corresponding to nonzero spectrum values (equality of the Riesz projections $p(0,a)$ and $p(0,b)$, if $0$ belongs to the spectrum, follows from the fact that the Riesz projections sum to $\mathbf 1$):

\begin{theorem}\label{char1} Let $A$ be a Banach algebra. If $a,b\in A$ and $\sigma(a)$ is finite, then $d(a,b)=0$ if and only if the following conditions hold:
	\begin{itemize}
		\item[(i)]{ $\sigma(a)=\sigma(b)$ }
		\item[(ii)]{ $p(\lambda,a)=p(\lambda,b)$ for each $\lambda\in\sigma(a)$.}
	\end{itemize}
\end{theorem}
\begin{proof}$\Rightarrow$:  Quasinilpotent equivalence implies (i). If $a$ and $b$ are quasinilpotent, then (ii) is trivial; so assume the contrary. Writing
	$\sigma(a)=\sigma(b)=\{\lambda_1,\dots,\lambda_n\}$ the Holomorphic Calculus implies
	\[a=\sum_{i=1}^nap_i\ \ \ \mbox{ and }\ \ \ \ b=\sum_{i=1}^nbq_i\] where $p_i=p(\lambda_i,a)$ and $q_i=p(\lambda_i,b)$. If we notice further, from the Spectral Mapping Theorem, that $\{(a-\lambda_i)p_i\}_{i=1}^n$ is a set of commuting quasinilpotents and similarly that $\{(b-\lambda_i)p_i\}_{i=1}^n$ is a set of commuting quasinilpotents, then it follows, using the fact that the Riesz projections sum to $1$,  that
	\begin{equation}\label{qn}
	\sigma(a-(\lambda_1p_1+\cdots+\lambda_np_n))=\{0\}=\sigma(b-(\lambda_1q_1+\cdots+\lambda_nq_n)).
	\end{equation}
	Since $d$ is a semimetric \eqref{qn} implies
	\begin{equation*}
	\begin{array}{ll}{} & d(\lambda_1p_1+\cdots+\lambda_np_n,\lambda_1q_1+\cdots+\lambda_nq_n)\\
	\leq & d(\lambda_1p_1+\cdots+\lambda_np_n,a)+d(a,b)+d(b,\lambda_1q_1+\cdots+\lambda_nq_n)\\
	= & 0.
	\end{array}
	\end{equation*}
	This implies $p_i=q_i$ for each $i$ by Lemma~\ref{Riesz}.
	
	$\Leftarrow$: If $a$, and hence $b$, are quasinilpotent the result follows trivially. Otherwise write $\sigma(b)=\sigma(a)=\{\lambda_1,\dots,\lambda_n\}$. As above, using (ii), we can write
	\[a=\sum_{i=1}^nap_i\ \ \ \mbox{ and }\ \ \ \ b=\sum_{i=1}^nbp_i\] where $p_i=p(\lambda_i,a)=p(\lambda_i,b)$. Again applying the Spectral Mapping Theorem to $\{(a-\lambda_i)p_i\}_{i=1}^n$ and  $\{(b-\lambda_i)p_i\}_{i=1}^n$, but now using \cite[Corollary 2.1]{razpet} and Proposition~\ref{quasi}, we see
	\begin{align*}
	d(a,b)&=d(ap_1+\cdots+ap_n,bp_1+\cdots+bp_n)\\
	&=d((a-\lambda_1)p_1+\cdots+(a-\lambda_n)p_n,(b-\lambda_1)p_1+\cdots+(b-\lambda_n)p_n)\\&=0.
	\end{align*}
\end{proof}

Theorem~\ref{char1} sheds light on the remark following Corollary~\ref{traceappl}: If $a\in\Soc(A)$ and $\lambda\in\sigma^\prime(a)$, then \cite[Theorem 2.6]{a+m} shows that $m(\lambda,a)=\rank p(\lambda,a)$.

\begin{theorem}\label{char2} Let $A$ be a Banach algebra. If $a,b\in A$ and $\sigma(a)$ is finite, then $d(a,b)=0$ if and only if
	$a-b$ can be expressed as the difference of two quasinilpotent elements, $r_a$ and $r_b$, commuting with $a$ and $b$ respectively.
\end{theorem}
\begin{proof}$\Rightarrow$: The proof of Theorem~\ref{char1} shows that we can write
	\[a=\sum_{i=1}^nap_i\ \ \ \mbox{ and }\ \ \ \ b=\sum_{i=1}^nbp_i\] where $p_i=p(\lambda_i,a)=p(\lambda_i,b)$. It is then elementary that
	\begin{equation}\label{diff}
	a-b=\sum_{i=1}^n(ap_i-\lambda_ip_i)-\sum_{i=1}^n(bp_i-\lambda_ip_i).
	\end{equation}
	Setting the first summation on the right equal to $r_a$, and the second to $r_b$, the result follows.
	
	$\Leftarrow$: With the hypothesis, and using Lemma~\ref{qc} in the end
	\[a-b=r_a-r_b\Rightarrow a-r_a=b-r_b\Rightarrow d(a-r_a,b-r_b)=0\Rightarrow d(a,b)=0.\]
\end{proof}

\begin{corollary}\label{char3}
	Let $A$ be a Banach algebra. If $a,b\in A$ and $\sigma(a)$ is finite, then $d(a,b)=0$ if and only if there exists $c\in A$, commuting with $a$ and $b$, such that both $a-c$ and $b-c$ are quasinilpotent. In particular, if $A$ is semisimple, and $a,b\in\Soc(A)$, then $c$ can be located in $\Soc(A)$, and quasinilpotency can be reduced to nilpotency.
\end{corollary}
\begin{proof}
	Suppose $d(a,b)=0$. Using Theorem~\ref{char2}, take $c=a-r_a=b-r_b$. Conversely, if $c$ commutes with $a$ and $b$, and both $a-c$ and $b-c$ are quasinilpotent, then write $a-b=(a-c)-(b-c)$; the result follows from  Theorem~\ref{char2}. For the second part, to find $c\in\Soc(A)$, use the fact that the Riesz projections corresponding to nonzero spectrum points of $a$ and $b$ also belong to $\Soc(A)$. Notice then that any quasinilpotent element in $\Soc(A)$ is in fact nilpotent.
\end{proof}

In \cite[Theorem 3.3]{raub} a special case of Theorem~\ref{char2} appears: $a\in\Soc(A)$ and $b\in A$ are quasinilpotent equivalent, and $a$ assumes its rank at $1$, a property which entails $ap_i=\lambda_ip_i$ for each Riesz projection $p_i=p(\lambda_i,a)$; however $b$ does not necessarily share the property, so that $a\not=b$ is possible. The most general conclusion here is that $a-b$ is a quasinilpotent element commuting with $a$ (in fact $a$ and $b$ commute). Of course, this is in accordance with Theorem~\ref{char2} above; and in fact the result is immediate from \eqref{diff}. It is further clear, as also shown in \cite[Theorem 3.2]{raub}, that if $b$ is in the socle and assumes its rank at $1$, then $a=b$. We conclude this section by extending the results in \cite{raub}, but now in a different direction. The arguments in \cite{raub} relied on Aupetit and Mouton's Diagonalization Theorem: An element $a\in\Soc(A)$ which assumes its rank at $1$ takes the form $a=\lambda_1p_1+\cdots+\lambda_np_n$ where $p_i=p(\lambda_i,p_i)$). The following generalization of \cite[Theorem 2.8]{a+m} was proved in \cite[Theorem 3.1]{brits}:

\begin{theorem}[Generalized Diagonalization Theorem]\label{GDT}
	Let $A$ be a semisimple Banach algebra and $0\not=a\in \Soc(A)$. Then $a$ is a linear combination of mutually orthogonal minimal idempotents if and only if $a$ assumes its rank at a commuting $y\in A$; that is, if and only if there exists $y\in A$ commuting with $a$ such that $\rank(a)=\#\sigma^\prime(ya)$.
\end{theorem}

\begin{theorem}\label{rudi}
	Let $A$ be a semisimple Banach algebra with $a,b\in  \Soc(A)$, and suppose $a$ and $b$ assume their respective ranks on $\comm(a)$ and $\comm(b)$. Then $\rho(a,b)=0$ if and only if $a=b$.
\end{theorem}
\begin{proof} By Theorem\,\ref{GDT} we can write
	\[a=\lambda_1 p_1+\ldots +\lambda_n p_n\]
	and
	\[b=\alpha_1 q_1+\ldots +\alpha_m q_m\]
	where the $p_i$ form a set of mutually orthogonal minimal projections, and the $\lambda_i$ are (not necessarily distinct) scalars. The same statement holds of course for the $q_i$ and $\alpha_i$. The $p_i$ and $q_i$ are not necessarily Riesz projections. However, if some $\lambda_i$ appears more than once in the above representation of $a$, then the total sum of the orthogonal minimal projections with coefficients equal to $\lambda_i$ gives the Riesz projection $p(\lambda_i,a)$. Theorem\,\ref{GDT} thus implies that the Riesz projections are mutually orthogonal, and hence, following the proof of Theorem~\ref{char2}, $r_a=0$. The same argument gives $r_b=0$, and $a=b$ follows from Theorem~\ref{char2}.
\end{proof}
In \cite[Theorem 3.3]{raub} the requirement that $a$ be maximal rank can be relaxed to the weaker ``$a$ assumes it rank on $\comm(a)$":
\begin{theorem}
	Let $A$ be a semisimple Banach algebra with $a\in\Soc (A)$, and suppose $a$ assumes its rank on $\comm (a)$. If $b\in A$, then $d(a,b)=0$ if and only if $a-b$ is quasinilpotent, and $a$ commutes with $b$.
\end{theorem}
\begin{proof}
	If $d(a,b)=0$, then $a-b=r_a-r_b$ with $r_a$ and $r_b$ as in Theorem~\ref{char2}. Since $a$ assumes its rank on $\comm (a)$, the argument in the proof of Theorem~\ref{rudi} shows that $r_a=0$. Hence $a-b=-r_b\in\QN(A)$. Since $r_b$ commutes with $b$ it also follows from $a-b=-r_b$ that $a$ and $b$ commute.
	Conversely, if $a-b$ is quasinilpotent, and  $a$ commutes with $b$, then $d(a,b)=r(a-b)=0$.
\end{proof}

\section{C*-algebras and quasinilpotent equivalence}\label{sec5}
In this section we investigate the effect of elements being quasinilpotent equivalent in C*-algebras.
The first result characterizes elements $b$ in a C*-algebra which are quasinilpotent equivalent to a normal element $a$ with finite spectrum:
\begin{theorem}\label{C*1}
	Let $A$ be a C*-algebra and let $a$ be a normal element of $A$ with finite spectrum. If $b\in A$ then $d(a,b)=0$ if and only if $a-b$ is quasinilpotent, and $a$ commutes with $b$.
\end{theorem}
\begin{proof}$\Leftarrow$: If $a-b\in\QN(A)$ and $ab=ba$, then, by the comment following \eqref{rhoab}, $d(a,b)=r(a-b)=0$. \\
	$\Rightarrow$: If $\sigma (a)=\{\lambda_1,\ldots ,\lambda_k\}$ then in view of \cite[Corollary 6.2.8]{aup} there exist self adjoint orthogonal projections $p_1,\ldots ,p_k$ in the commutative closed subalgebra generated by $1$, $a$ and $a^*$, such that $p_1+\ldots +p_k=1$ and $a={\lambda_1}p_1+\ldots +{\lambda_k}p_k$. Moreover, the proof of the aforementioned result shows that $p_j=p(\lambda_j,a)$ for each $j$. Since the projections are orthogonal, $ap_i={\lambda_i}p_i$ for $1\leq i\leq k$. The result is then clear from the comments following Corollary~\ref{char3}.
\end{proof}

\begin{theorem}
	Let $A$ be a C*-algebra. If $a$ and $b$ are normal elements of $A$, and if $0$ is the only possible accumulation point of $\sigma(a)$, then $d(a,b)=0$ if and only if $a=b$.
\end{theorem}
\begin{proof} If $\sigma(a)$ is finite, then the result follows as a special case of Theorem~\ref{C*1}. So we may assume $\sigma(a)=\sigma(b)=\{\lambda_1,\lambda_2,\dots\}\cup\{0=\lambda_0\}$ where $(\lambda_j)$ is a sequence of nonzero complex numbers converging to $0$. For each $n\in\mathbb N$ define $f_n(\lambda)$ and $h_n(\lambda)$ on $\sigma(a)$, by
\begin{displaymath}
f_n(\lambda_j)=\left\{\begin{array}{ll}1 & 0<j\leq n\\
0 & j>n \mbox{ or } j=0.\end{array}\right.
\end{displaymath}
and
\begin{displaymath}
h_n(\lambda_j)=\left\{\begin{array}{ll}0 & 0<j\leq n\\
1 & j>n \mbox{ or } j=0.\end{array}\right.
\end{displaymath}
Using the Continuous Functional Calculus throughout the remainder of this proof we have:
\[f_n(a)=p(\lambda_1,a)+\cdots+p(\lambda_n,a)\]
and
\[f_n(b)=p(\lambda_1,b)+\cdots+p(\lambda_n,b).\]
Notice further that, for each $i\in\mathbb N$, $ap(\lambda_i,a)-\lambda_ip(\lambda_i,a)$ is quasinilpotent and normal, from which it follows that
$ap(\lambda_i,a)=\lambda_ip(\lambda_i,a)$. Similarly $bp(\lambda_i,b)=\lambda_ip(\lambda_i,b)$. For each $n\in\mathbb N$ we can write
\[a=af_n(a)+ah_n(a)\]
and
\[b=bf_n(b)+bh_n(b).\]
But now, using Lemma~\ref{Riesz}, we see that $af_n(a)=bf_n(b)$ for each $n\in\mathbb N$. Finally, from the fact that
$ah_n(a)$ and $bh_n(b)$ are both normal, we have
\begin{align*}\|a-b\|&=\|ah_n(a)-bh_n(b)\|\\&
\leq \|ah_n(a)\|+\|bh_n(b)\|\\&
=r(ah_n(a))+r(bh_n(b))\\&
=2\max\{|\lambda_j|:j>n\}
\end{align*}
holds for all $n\in\mathbb N$. Letting $n\rightarrow\infty$ we obtain $\|a-b\|=0$.
\end{proof}

\bibliographystyle{amsplain}

\end{document}